\documentclass[a4paper, 11pt, underruledhead]{paper}

\usepackage{math, enumerate}
\usepackage[mathscr]{euscript}
\usepackage[dvips]{graphicx}

\usepackage[usenames,dvipsnames]{color}

\def\Lines{\mathcal{L}}

\DeclareMathOperator{\GF}{GF}
\DeclareMathOperator{\impcoll}{L^\infty}

\def\LineOn(#1,#2){\overline{{#1},{#2}\rule{0em}{1,5ex}}}

\def\penc{{\bf p}}

\def\starof{\mathrm{S}}
\def\topof{\mathrm{T}}

\def\lines{{\cal L}}
\def\peki{{\mathcal{P}}}

\def\AffineSpSymb{\mathbf{A}}
\def\AfSpace(#1){\ensuremath{\AffineSpSymb(#1)}}

\def\PencSpace(#1,#2){{\bf P}_{#2}({#1})}

\def\SpineSp(#1,#2,#3,#4){\ensuremath{\AffineSpSymb_{#1, #2}(#3, #4)}}

\def\fixproj{\ensuremath{\goth P}}
\def\fixout{\ensuremath{\goth D}}
\def\fixoutf(#1,#2){\mathord{\fixout_{#1}(#2)}}

\def\adjac{\mathrel{\sim}}
\def\impadjac{{\adjac^\infty}}

\def\linesout{{\Lines_{\W}}}
\def\pointsout{{S_{\W}}}
\def\indf(#1,#2){\ind_{#1}(#2)}
\def\fixind{{\indf(\fixproj,\W)}}
\def\pplanes{\mathcal{H}}
\let\polarity\pi

\def\A{{\mathfrak A}}

\def\M{{\mathfrak M}}

\def\W{{\cal W}}
\def\rlines{{\mathcal K}}
\def\pps{{\mathfrak G}}
\def\ppslines{{\mathcal G}}

\def\pclique{{\mathcal S}}
\def\stardir(#1){\pclique_\parallel(#1)}
\def\stardirin(#1,#2){\pclique_\parallel^{#2}(#1)}
\def\uzg{\mathrel{\sigma}}

\newenvironment{cmath}{%
  \par
  \smallskip
  \centering
  $
}{%
  $
  \par
  \smallskip
  \csname @endpetrue\endcsname
}

\newenvironment{ctext}{%
  \par
  \smallskip
  \centering
}{%
 \par
 \smallskip
 \csname @endpetrue\endcsname
}

\begin{document}

\title[The complement of a point subset...]{The complement of a point subset in a projective space and a Grassmann space}

\author{K. Petelczyc and M. \.Zynel}

%\supportauthor{K. Pra\.zmowski}

\maketitle

\begin{abstract}
  In a projective space we fix some set of points, a horizon, and investigate the
  complement of that horizon. 
  We prove, under some assumptions on the size of lines, that the ambient
  projective space, together with its horizon, both can  be recovered in that
  complement.
  Then we apply this result to show something similar for Grassmann spaces.
\end{abstract}

\begin{flushleft}\small
\end{flushleft}

%% Section %%%%%%%%%%%%%%%%%%%%%%%%%%%%%%%%%%%%%%%%%%%%%%%%%%%%%%%%%%%%%%%%%%%%%

\section{Introduction}

\subsection{Our goal}

Two definitions of an affine space are known: it is a projective space with one
of its hyperplanes removed, or it is a projective space with one of its
hyperplanes distinguished. It is also known that these two definitions are
equivalent, since the removed hyperplane can be recovered in terms of so
obtained affine space. A general question arises: how big the remaining fragment
of a projective space must be so as the surrounding space can be recovered in
terms of the internal geometry of this fragment? Analogous question could be
asked in the context of projective Grassmannians. In this paper we propose a
solution involving parallelism, imitating the affine parallelism, valid under
certain assumptions on the size of improper parts of lines.  Our solution is
just {\emph a} solution: there are fragments of projective spaces (e.g.  ruled quadrics)
which do not satisfy our requirements and from which the surrounding spaces are
recoverable, but with completely different methods.

\subsection{Motivations and references}

In this paper we deal with a point-line space $\A$ where a subset $\W$ of
points, a horizon, and those lines that are entirely contained in $\W$ are
removed. We call such a reduct of $\A$ the complement of $\W$ in $\A$ and denote
it by $\fixoutf(\A,\W)$. 
As for which lines to remove from $\A$ our approach is not unique and
there are different complements considered in the literature, e.g.\ in \cite{sulima},
where all the lines that meet $\W$ are deleted.

The most classic and vivid example of a complement $\fixoutf(\A,\W)$, aforementioned at the beginning, 
is an affine space i.e.\ the
complement of a hyperplane $\W$ in a projective space $\A$.   
There are more examples where removing a hyperplane is also successful and where in result
we get some affine geometries.
In a polar space deleting its geometric hyperplane yields an affine polar space 
(cf. \cite{cohenshult}, \cite{afpolar}, \cite{csaffpolar}).
In a Grassmann space, in other words, in a space of pencils, this construction produces an 
affine Grassmannian (cf. \cite{cuypers}).
When $\W$ is any subspace, not necessarily a hyperplane, in a projective space $\A$, then
so called slit space arises as the complement $\fixoutf(\A,\W)$
possessing both projective and affine properties (cf. \cite{KM67}, \cite{KP70}).

In \cite{kradzisz} (see also \cite{kradzisz2}) there is introduced an axiom
system of semiaffine partial linear spaces (SAPLS in short) that are,
without going into details, weak geometries with parallelism. It is astonishing
that the set of directions of lines in a SAPLS form a subspace $\W$ of some
projective space $\A$ so, we can treat this SAPLS as the complement 
$\fixoutf(\A,\W)$. More precisely, it is proved in \cite{kradzisz}, that the
class of semiaffine linear spaces coincides with the class of forementioned slit
spaces. In other words, semiaffine linear spaces are slit spaces with
parallelism that is not necessarily reflexive.

Another remarkable example is a spine space (cf. \cite{spinesp},
\cite{autspine}), that emerges from a Grassmann space over a vector space by
taking only those points which as $k$-dimensional vector subspaces meet a fixed
vector subspace $W$ in a fixed dimension $m$. In case $m=0$ and $k$ is the
codimension of $W$ we get a pretty well known space of linear complements (cf.
\cite{lincomp} and also \cite{blunck-havlicek}).

Now we turn back to our major question: is it possible to recover the underlying
space $\A$ and the horizon $\W$ in the complement $\fixoutf(\A,\W)$? This is not
a completely new question and there are some papers devoted to such recovery
problem.
In \cite{slitgras} projective Grassmannians are successfully recovered from
complements of  their Grassmann substructures.
The concept of two-hole slit space is introduced in \cite{sulima}.  It is a
point-line space whose point set is the complement of the set of points of two
fixed complementary subspaces, not hyperplanes, in a projective space and the line set is the set
of all those lines which do not intersect any of these two subspaces.
This is not exactly what we have used to call a complement as lines are taken differently. 
Nevertheless,
using very similar methods to ours the recovery problem has been solved here
in incidence geometry settings. It has been also solved 
in chain geometry associated with a linear group approach. 
In \cite{partaffine} a partial linear space  with parallelism
embeddable in an affine space with the same point set is studied. On the other hand, it can be
considered as an affine space with some lines deleted. As it was mentioned, an affine space
is a complement of a projective space, and thus there is a correspondence between \cite{partaffine} and
our research. To restore missing lines (and preserve parallelism) the authors introduce some additional
combinatorial condition. Namely, all investigations in \cite{partaffine} are done under assumption,
that on every line $k$, for any point $p\notin k$,
there is more points collinear with $p$ than those not collinear with $p$.
Likewise, we require that on every line of $\fixoutf(\A,\W)$ there is less points of $\W$
than the other. 

\subsection{Results}

The purpose of this paper is to give an answer to the above question for projective spaces
and Grassmann spaces where a set of points, under some restrictions, has been removed.
In the first part we focus on a projective space as $\A$. Purely in terms of
the complement of $\W$ in $\A$ we reconstruct the removed points and lines as well
as the relevant incidence. An affirmative answer to our question is given in Theorem~\ref{thm:main-proj}.

In the second part we investigate a Grassmann space as $\A$.
It is known that every maximal strong subspace of a Grassmann space is a projective space.
This, together with some additional assumptions, lets us apply the result obtained for projective spaces 
to recover the whole $\A$ and its horizon $\W$ in terms of the complement $\fixoutf(\A,\W)$.
That way another affirmative answer to our question is formulated in
Theorem~\ref{thm:main-grass}.

At the end we give some examples where our results can be applied or not.

%% Section %%%%%%%%%%%%%%%%%%%%%%%%%%%%%%%%%%%%%%%%%%%%%%%%%%%%%%%%%%%%%%%%%%%%%

\section{Generalities}

A point-line structure $\A=\struct{S, \lines}$, where the elements of $S$ are
called \emph{points}, the elements of $\lines$ are called \emph{lines}, and
where $\lines\subset2^S$, is said to be a \emph{partial linear space}, or 
a \emph{point-line space}, if two
distinct lines share at most one point and every line is of size (cardinality) at least 2 (cf. \cite{cohen}).
The incidence relation between points and lines is basically the membership
relation $\in$. 

For distinct points $a, b\in S$ we say that they are \emph{adjacent} (\emph{collinear}) and write
$a\adjac b$ if there is a line in $\lines$ through $a, b$.
The line through distinct points $a, b$ will be written as  $\LineOn(a, b)$.
For distinct lines $k, l\in\lines$ we say that they are \emph{adjacent} and
write $k\adjac l$ if they intersect in a point from  $S$.
We say that three pairwise distinct points $a_1, a_2, a_3\in S$ (or lines $k_1, k_2, k_3\in\lines$)
form a \emph{triangle} in $\A$ if they are pairwise adjacent and not collinear (or not concurrent
respectively), which we write as $\triangle(a_1, a_2, a_3)$
(or $\triangle(k_1, k_2, k_3)$ respectively).
The points $a_1, a_2, a_3$ are called vertices and the lines $k_1, k_2, k_3$
are called sides of the triangle.
A \emph{subspace} of $\A$ is any set $X\subseteq S$ with the property that
every line which shares with $X$ two or more points is entirely contained in $X$.
We say that a subspace $X$ of $\A$ is \emph{strong} if any two points in $X$ are
collinear. 
If $S$ is strong then $\A$ is said to be a \emph{linear space}.

%% Subsection %%%%%%%%%%%%%%%%%%%%%%%%%%%%%%%%%%%%%%%%%%%%%%%%%%%%%%%%%%%%%%%%%%

\subsection{Complements}

Let us fix a subset $\W\subset S$.
By the \emph{complement of\/ $\W$ in $\A$} we mean the structure
  $$\fixoutf(\A,\W) := \struct{\pointsout, \linesout},$$
where
  $$\pointsout := S\setminus\W \qquad\text{and}\qquad
     \linesout := \{ k\in\lines\colon k\nsubseteq\W\}.$$
Here, the incidence relation is again $\in$, inherited from $\A$, but
limited to the new point set and line set. It should not lead to confusion,
as it will be clear from context which incidence is which.
Following a standard convention we call the points and lines of the
complement $\fixoutf(\A,\W)$ \emph{proper}, and those points and lines
of $\A$ that are not in $\pointsout$, $\linesout$ respectively are said to be 
\emph{improper}.
The set $\W$ will be called the \emph{horizon of\/ $\fixoutf(\A,\W)$}.
Among proper lines of $\A$ we distinguish those which intersect the horizon $\W$
and call them \emph{affine}.
A triangle in $\A$ is said to be \emph{proper} if all its vertices and
sides are proper.

Let us introduce, quite critical for our study, the notion of
\emph{the index of\/ $\W$ in $\A$}.
It will be denoted by $\indf(\A,\W)$ and
\begin{ctext}
  $\indf(\A,\W)=\lambda$ iff there exist a line with $\lambda$ improper points\\
  and every line either has at most $\lambda$ improper points or it is contained in $\W$.
\end{ctext}
For $\indf(\A,\W) = 0$ we get $\fixoutf(\A,\W)=\A$. We assume in the sequel that
\begin{equation}\label{eq:non-zero-ind}
  0 < \indf(\A,\W) < \infty
\end{equation}
and $\indf(\A,\W)$ is a well defined integer.
There is one more caveat yet. The complement $\fixoutf(\A,\W)$ is not a partial linear
space in general as there could be one-element sets in $\linesout$.
This defect is ruled out by the assumption that
\begin{equation}\label{eq:size-of-line}
  \text{\itshape every line of\/ $\A$ has size at least\/ $2\indf(\A,\W)+2$},
\end{equation}
which is obviously too strong for this specific purpose but will become
essential later.
Now, the very first observation.

\begin{fact}\label{fact:basic}\strut
  \begin{sentences}
  \item\label{fact:basic:proper}
    Every line through a proper point is proper.
  \item\label{fact:basic:min}
    There are at least $\fixind +2$ proper points on a proper line.
  \item\label{fact:basic:notall}
    $\W \neq S$.
  \end{sentences}
\end{fact}

The consequence of \ref{fact:basic}\eqref{fact:basic:min} is that
there are at least 3 proper points on every line of the complement
$\fixoutf(\A,\W)$, thus it is a partial linear space.

%% Subsection %%%%%%%%%%%%%%%%%%%%%%%%%%%%%%%%%%%%%%%%%%%%%%%%%%%%%%%%%%%%%%%%%%

\subsection{Grassmann spaces}

Let $V$ be a vector space of dimension $n$ with $3\le n<\infty$. The set of all
subspaces of $V$ will be written as $\Sub(V)$ and the set of all $k$-dimensional
subspaces (or $k$-subspaces in short) as $\Sub_k(V)$. The most basic concept for
us here is a \emph{$k$-pencil} that is the set of the form
  $$\penc(H, B) := \{ U\in\Sub_k(V)\colon H\subset U\subset B\},$$
where $H\in\Sub_{k-1}(V)$, $B\in\Sub_{k+1}(V)$, and $H\subset B$.
The family of all such $k$-pencils will be denoted by $\peki_k(V)$.
We consider a \emph{Grassmann space} (also known as a \emph{space of pencils})
  $$\M = \PencSpace(V, k) = \struct{\Sub_k(V), \peki_k(V)},$$
a point-line space with $k$-subspaces of $V$ as points and $k$-pencils as lines
(see \cite{polargras}, \cite{slitgras} for a more general definition, see also \cite{mark}).
For $0 < k < n$ it is a partial linear space. For $k=1$ and
$k=n-1$ it is a projective space. So we assume that
  $$1< k < n-1.$$

It is known that there are two classes of maximal strong subspaces in $\M$:
\emph{stars} of the form
  $$\starof(H) = [H)_k = \{U\in\Sub_k(V)\colon H\subset U\},$$
where $H\in\Sub_{k-1}(V)$, and \emph{tops} of the form
  $$\topof(B) = (B]_k = \{U\in\Sub_k(V)\colon U\subset B\},$$
where $B\in\Sub_{k+1}(V)$.
Although non-maximal stars $[H, Y]_k$ and non-maximal tops $[Z, B]_k$, for
some $Y, Z\in\Sub(V)$, make sense but in this paper when we say `a star' or `a top'
we mean a maximal strong subspace.
It is trivial that every line, a $k$-pencil $p=\penc(H, B)$, of $\M$ can be 
uniquely extended to the star $\starof(p):=\starof(H)$ and to the top 
$\topof(p):=\topof(B)$.

Numerous intrinsics properties of Grassmann spaces are very well known  or can
be easily verified (cf. \cite{mark}). Let us recall some of them and prove a
couple of technical facts that will be needed later.

\begin{fact}\label{fact:lineinstartop}
  Let $S$ be a star and $T$ be a top in $\M$. If\/ $S\cap T\neq\emptyset$, then
  $S\cap T$ is a line and if additionally $U_1\in S\setminus T$, $U_2\in
  T\setminus S$, then $U_1, U_2$ are not collinear in $\M$.
\end{fact}

\begin{cor}\label{cor:starcaptop}
  Let $S$ be a star and $T$ be a top in $\M$. If\/ $U_1\in S$, $U_2\in T$, and\/
  $U_1\adjac U_2$, then either $S\cap T = \emptyset$, or $U_1\in S\cap T$, or
  $U_2\in S\cap T$.
\end{cor}

We say that a subspace of $\M$ is a \emph{plane} if it is (up to an isomorphism)
a projective plane.
It is known that every strong subspace of $\M$ carries the structure of a
projective space. For a star $S = [H)_k$ a plane contained in $S$ is of the form
$[H, Y]_k$, where $Y\in\Sub_{k+2}(V)$, while for a top $T = (B]_k$ a plane contained
in $T$ is of the form $[Z, B]_k$, where $Z\in\Sub_{k-2}(V)$.

\begin{lem}\label{lem:top-star-top}
  Let\/ $U$ be a point and\/ $T_i = (B_i]_k$, $i=1,2$\/ be distinct tops of\/ $\M$
  with $U\in T_1\cap T_2$.
  \begin{sentences}
  \item
    If\/ $\Pi = [Z, B_1]_k$ is a plane through $U$ contained in $T_1$, then
      $$\{ \starof(p)\cap T_2\colon p\in\peki_k(V),\ U\in p\subset \Pi\}$$
    is the set of lines through $U$ on the plane $[Z, B_2]_k$ contained in $T_2$.
    Consequently, we have a bijection $f$ that maps planes through $U$ contained in $T_1$
    to planes through $U$ contained in $T_2$.
  \item
    Let\/ $\Pi_i$ for $i=1,2,3$ be pairwise distinct planes through $U$ contained in $T_1$. 
    If\/ $\Pi_1\cap\Pi_2\cap\Pi_3$ is a line, then $f(\Pi_1)\cap f(\Pi_2)\cap f(\Pi_3)$ is a line.
  \end{sentences}
\end{lem}

\begin{proof}
  Note first that $T_1\cap T_2 = \{U\}$ and $U=B_1\cap B_2$. 
  
  (i):
  For the plane $\Pi$ observe that $Z\in\Sub_{k-2}(V)$ and $Z\subset U$. We have
  \begin{multline*}
    X := \{ \starof(p)\cap T_2\colon p\in\peki_k(V),\ U\in p\subset \Pi\} = \\
      \{ \starof(p)\cap T_2\colon p = \penc(H, B_1),\ H\in[Z, U]_{k-1} \} =
        \{ \penc(H,B_2)\colon H\in[Z, U]_{k-1} \}
  \end{multline*}
  as $H\subset U\subset B_2$. So, $X$ is the set of lines through $U$ on the
  plane $[Z, B_2]_k$ contained in $T_2$, which completes this part of the proof.
  
  (ii):
  Let $\Pi_i = [Z_i, B_1]_k$, where $Z_i\in\Sub_{k-2}(V)$ and $Z_i\subset U$ for $i=1,2,3$.
  By assumption $H := Z_1+Z_2+Z_3\in\Sub_{k-1}(V)$ and 
    $$\Pi_1\cap\Pi_2\cap\Pi_3 = \penc(H, B_1).$$
  Since $f(\Pi_i) = [Z_i, B_2]_k$ for $i=1,2,3$ and $H\subset U\subset B_2$,
  we get that 
    $$f(\Pi_1)\cap f(\Pi_2)\cap f(\Pi_3) = \penc(H, B_2)$$
  and we are through.
\end{proof}

Dually to \ref{lem:top-star-top} we have the following for stars.

\begin{lem}
  Let\/ $U$ be a point and\/ $S_i = [H_i]_k$, $i=1,2$\/ be distinct stars of\/ $\M$
  with $U\in S_1\cap S_2$.
  \begin{sentences}
  \item
    If\/ $\Pi = [H_1, Y]_k$ is a plane through $U$ contained in $S_1$, then
      $$\{ \topof(p)\cap S_2\colon p\in\peki_k(V),\ U\in p\subset \Pi\}$$
    is the set of lines through $U$ on the plane $[H_2, Y]_k$ contained in $S_2$.
    Consequently, we have a bijection $f$ that maps planes through $U$ contained in $S_1$
    to planes through $U$ contained in $S_2$.
  \item
    Let\/ $\Pi_i$ for $i=1,2,3$ be pairwise distinct planes through $U$ contained in $S_1$. 
    If\/ $\Pi_1\cap\Pi_2\cap\Pi_3$ is a line, then $f(\Pi_1)\cap f(\Pi_2)\cap f(\Pi_3)$ is a line.
  \end{sentences}
\end{lem}

%% Section %%%%%%%%%%%%%%%%%%%%%%%%%%%%%%%%%%%%%%%%%%%%%%%%%%%%%%%%%%%%%%%%%%%%%

\section{Complements in projective spaces}\label{sec:proj}

Let $\fixproj=\struct{S, \lines}$ be a projective space of dimension at least $3$
and let $\W\subset S$ such that \eqref{eq:size-of-line} is satisfied.
We will investigate the complement of\/ $\W$ in $\fixproj$,
that is the structure $\fixout := \fixoutf(\fixproj,\W)$.
 
Note that in the case where $\W$ is a hyperplane, the complement
$\fixout$  is an affine space, while if $W$ is not necessarily a hyperplane but
a proper subspace of $\fixproj$, then $\fixout$ is a slit space
(cf. \cite{KM67}, \cite{KP70}, \cite{slitgras}).

The goal of the first part of our paper is as follows.
\begin{thm}\label{thm:main-proj}
  If\/ $\fixproj$ is a projective space of dimension at least $3$ and $\W$ is its
  point subset satisfying \eqref{eq:size-of-line}, then both $\fixproj$ and\/ 
  $\W$ can be recovered in the complement\/ $\fixoutf(\fixproj,\W)$.
\end{thm}

More in vein of Klein's Erlangen Program this theorem can be rephrased
to the language of collineations.

\begin{thm}
  Under assumptions of \ref{thm:main-proj}
  for every collineation $f$ of the complement\/ $\fixoutf(\fixproj,\W)$ there is 
  a collineation $F$ of\/ $\fixproj$, leaving $\W$ invariant, such that
  $f = F|\pointsout$.
\end{thm}

%% Subsection %%%%%%%%%%%%%%%%%%%%%%%%%%%%%%%%%%%%%%%%%%%%%%%%%%%%%%%%%%%%%%%%%%

\subsection{Parallelism}

In affine structures parallelism is a key notion for projective completion.
We are going to use it here in a similar way.

Usually, two lines are said to be
parallel if they meet on the horizon. Following this idea we define \emph{parallelism}
$\parallel$ so that for lines $k_1, k_2\in\lines$
\begin{equation}\label{eq:parallel}
  k_1\parallel k_2 :\iff \emptyset\neq k_1\cap k_2\subseteq\W.
\end{equation}
In case $\fixout$ is a slit space our requirement
\eqref{eq:size-of-line} that lines are of size at least 3 makes good sense as it
is known that in an affine space  with lines of size 2 parallelism cannot be
defined in terms of incidence.

Definition \eqref{eq:parallel} involves
the incidence relation of the ambient projective space $\fixproj$.
To give an internal definition of our parallelism, expressible purely in terms of the
complement $\fixout$, we begin by proving a variant of Veblen (Pasch)
axiom for $\fixout$.

\begin{lem}\label{prop:veblen}
  If proper lines $k_1, k_2, k_3$ form a triangle in $\fixproj$ with at least
  one improper vertex, then there is a proper line that intersects
  $k_1, k_2, k_3$ in pairwise distinct proper points.
\end{lem}

\begin{proof}
  Let $a_m\in k_i\cap k_j$, where $\{m,i,j\} = \{1,2,3\}$, be the vertices of our triangle.
  Without loss of generality we can assume that the point $a_3$ is improper.
  By \ref{fact:basic}\eqref{fact:basic:min}
  we can take proper points $a\in k_1$ and $b_1,\dots, b_{\fixind+1}\in k_2$
  such that $a\neq a_2$ and $b_i\neq a_1$ for all $i=1,\dots,\fixind+1$.
  Consider the lines $l_i := \LineOn(a, b_i)$ for $i=1,\dots,\fixind+1$.
  All these lines are proper and all they intersect the line $k_3$. As there are only
  up to $\fixind$ improper points on $k_3$, one of $l_i$ is the required line.
\end{proof}

Now, we are able to give a definition of $\parallel$ purely in terms of
$\fixout$.

\begin{prop}\label{prop:parallel}
  Let $k_1, k_2\in\linesout$ with $k_1\neq k_2$. Then
  \begin{multline}\label{parallel:veblen}
    k_1\parallel k_2 :\iff (\exists k_3, l\in\linesout) (\exists c_1, c_2, c_3\in\pointsout)
      \bigl[\, \triangle(k_1, k_2, k_3) \Land \\
         c_1\in l, k_1 \Land c_2\in l, k_2 \Land c_3\in l, k_3 \,\bigr]
           \Land (\nexists\; x\in\pointsout)[\, x\in k_1, k_2 \,].
  \end{multline}
\end{prop}

\begin{proof}
  \ltor
  Since $k_1, k_2$ are parallel we have $\emptyset\neq k_1\cap k_2\in\W$.
  So, $k_1, k_2$ share an improper point. Take two proper points
  $a_1\in k_1$ and $a_2\in k_2$. They must be distinct and the line $k_3 := \LineOn(a_1, a_2)$
  is proper. We have $\triangle(k_1, k_2, k_3)$ with one improper vertex. Hence, by
  \ref{prop:veblen} there is a required line $l\in\linesout$ and points $c_1, c_2, c_3\in\pointsout$.

  \rtol
  Note that the lines $k_1,k_2$ intersect in an improper point, i.e.\
  $\emptyset\neq k_1\cap k_2\in\W$. So, $k_1\parallel k_2$
  by \eqref{eq:parallel}.
\end{proof}

Thanks to \ref{prop:parallel} we can distinguish affine lines using
the internal language of $\fixout$.

\begin{fact}\label{fact:affinelines}
  A line $k\in\linesout$ is affine iff
  there exists a line $l\in\linesout$ such that $l\neq k\parallel l$.
\end{fact}

%% Subsection %%%%%%%%%%%%%%%%%%%%%%%%%%%%%%%%%%%%%%%%%%%%%%%%%%%%%%%%%%%%%%%%%%

\subsection{Planes}

By analogy to proper line we introduce the term \emph{proper plane}. Namely we
 say that a plane $\Pi$ of $\fixproj$ is proper (or is a
\emph{plane of\/ $\fixout$}) if $\Pi\not\subset\W$. In the following lemma
we try to justify this terminology.

\begin{lem}\label{lem:triangle-plane}
  There is a proper triangle on every plane of\/ $\fixout$.
\end{lem}

\begin{proof}
  Let $\Pi$ be a plane of $\fixout$.
  There is a proper point $a$ on $\Pi$, since $\Pi\not\subset\W$.
  Every line through $a$ is proper by \ref{fact:basic}\eqref{fact:basic:proper}.
  There are at least $2\fixind +2$ lines
  through $a$ on $\Pi$ by \eqref{eq:size-of-line}.
  So, take two distinct lines $l_1,l_2$ with $a\in l_1,l_2\subset\Pi$, and
  consider two points $b,c\neq a$ such that $a\in l_1$, $b\in l_2$.
  Then $a,b,c$ and $l_1, l_2, \LineOn(b,c)$ form the required triangle.
\end{proof}

By \ref{lem:triangle-plane} we have a proper triangle on a plane $\Pi$ of $\fixout$.
Now we prove that this triangle spans $\Pi$.

\begin{lem}\label{triangle-span}
  Let\/ $\Pi$ be a proper plane and let $k_1, k_2, k_3$ be the sides of a
  proper triangle on\/ $\Pi$.
  For every (proper or improper) point $a\in\Pi$ there is a proper line through $a$ that intersects two of
  the sides $k_1,k_2,k_3$ in two distinct proper points.
\end{lem}

\begin{proof}
  Let $a_1, a_2, a_3$ be the vertices of our triangle of sides $k_1, k_2, k_3$ with
  $a_i\notin k_i$ for $i = 1, 2, 3$ and let $a\in\Pi$.
  If $a\in k_i$ for some $i=1,2,3$, then $k_i$ is the required line. So, assume that
  $a\notin k_1,k_2,k_3$.

  By \ref{fact:basic}\eqref{fact:basic:min} there are at least $\fixind$ proper points
  on $k_1$ distinct from $a_2, a_3$. Let $p_1,\dots, p_\lambda$ be such points
  where $\fixind\le\lambda$. Take $l_i := \LineOn(a, p_i)$ for $i=1,\dots, \lambda$.
  Each line $l_i$ is proper and intersect $k_2$ in some point
  $q_i$. Note that all these
  points $q_i$ are pairwise distinct.
  If $\fixind < \lambda$, then there is $j=1,\dots, \lambda$ such that
  $q_j$ is a proper point as by \eqref{eq:size-of-line} there are only up to
  $\fixind$ improper points on $k_2$. Thus, the line $l_j$ is the required one and
  we are through.
  If $\fixind = \lambda$, then either $q_j$ is a proper point for some
  $j=1,\dots, \lambda$ and $l_j$ is the required line, or all the points $q_i$,
  and only those, are improper points of $k_2$. In the latter case
  take $l := \LineOn(a, a_2)$. It is a proper line as $a_2$ is proper
  and it intersect
  $k_2$ in a point distinct from $q_i,\dots, q_\lambda$ thus, in a proper point
  as required.
\end{proof}

Immediately from \ref{triangle-span} we get

\begin{cor}\label{cor:planes}
  Proper planes are definable in terms of\/ $\fixout$.
\end{cor}

Next lemmas may look technical and superfluous until Section~\ref{sec:cliques}.

\begin{lem}\label{lem:bundle}
  There are at least three pairwise distinct non-coplanar proper lines through
  every (improper) point of\/ $\fixproj$.
\end{lem}

\begin{proof}
  Let $a$ be a point of $\fixproj$. If $a$ is proper, then the claim follows directly
  by \ref{fact:basic}\eqref{fact:basic:proper}.
  So, assume that $a$ is improper. By \ref{fact:basic}\eqref{fact:basic:notall}
  there is some proper point $a_1$, thus $a\neq a_1$. Every line through $a_1$
  is proper so, it is $k_1 := \LineOn(a, a_1)$. Take another line $l$ through $a_1$
  distinct from $k_1$.
  By \ref{fact:basic}\eqref{fact:basic:min} there is a proper point $a_2\in l$
  with $a_1\neq a_2$. The line $k_2 := \LineOn(a, a_2)$ is proper, $k_1\neq k_2$,
  and $a\in k_1, k_2$. So, we have a proper plane $\Pi = \gen{k_1, k_2}$.
  Now, as $\fixproj$ is at least $3$ dimensional, take a point $p$ not on $\Pi$.
  Regardless of whether $p$ is proper or not, there is a proper point
  $a_3$ on the line $\LineOn(p, a_1)$ with $a_3\neq a_1$ (possibly $a_3 = p$)
  by \ref{fact:basic}\eqref{fact:basic:min}.
  Lines $k_1$, $k_2$, $k_3 := \LineOn(a, a_3)$ are the required lines.
\end{proof}

A straightforward outcome of \ref{lem:bundle} is as follows.

\begin{cor}\label{lem:cosik}
  \strut
  \begin{sentences}
  \item\label{lem:cosik:a}
    There is a proper line through every point of\/ $\fixproj$.
  \item\label{lem:cosik:b}
    There is a proper plane through every point of\/ $\fixproj$.
  \end{sentences}
\end{cor}

\begin{lem}\label{lem:affine-lines}
  There are at least $\fixind+2$ affine lines through every improper point on a proper plane.
\end{lem}

\begin{proof}
  Consider an improper point $a$ and a plane $\Pi$ of $\fixout$ through $a$.
  By \eqref{eq:size-of-line} there are at least  $2\fixind+2$ lines through $a$ on $\Pi$.
  Assume that there are $\lambda$ improper lines $k_1,\dots,k_\lambda$ contained in $\Pi$
  through $a$. Take a proper point $b$ of $\Pi$ and a point $a_i\in k_i$ for some
  $i=1,\dots,\lambda$ distinct from $a$. The line $l:=\LineOn(a_i, b)$ is proper
  by \ref{fact:basic}\eqref{fact:basic:proper}, and intersects every of $k_1,\dots,k_\lambda$.
  There are up to $\fixind$ improper points on $l$, thus $\lambda\leq \fixind$. The remaining lines through
  $a$ are affine.
\end{proof}

\begin{lem}\label{impintersec}
  For every improper line $k$ there are at least\/ $\fixind+2$ proper planes containing $k$.
\end{lem}

\begin{proof}
  Consider the bundle $\cal F$ of all the planes in $\fixproj$ containing the line $k$
  and let $a\in k$. By \ref{lem:bundle} we have three proper lines $k_1, k_2, k_3$ through 
  $a$ that span three pairwise distinct proper planes. 
  Note that the line $k$, as an improper one, can be contained in  at most one of
  those planes. Take $\Pi$ to be one of the other two, thus $\Pi$ is a proper plane with
  $a\in\Pi$ and $k\not\subset\Pi$.
  
  Every plane from the family $\cal F$ intersects $\Pi$ in a line of $\fixproj$.
  All these lines form a pencil through $a$ on $\Pi$.  There are at least $\fixind+2$ proper 
  lines in that pencil by \ref{lem:affine-lines}. Each of these proper lines
  together with $k$ span a proper plane from $\cal F$ which completes the proof.
\end{proof}

%% Subsection %%%%%%%%%%%%%%%%%%%%%%%%%%%%%%%%%%%%%%%%%%%%%%%%%%%%%%%%%%%%%%%%%%

\subsection{Cliques of parallelism}\label{sec:cliques}

Let $\pclique$ be a maximal clique of $\parallel$ in $\fixout$. There are two possibilities:
\begin{sentences}
\item
  there is an improper point $a$ such that $a\in k$ for all $k\in\pclique$, or
\item
  there is a plane $\Pi$ in $\fixout$ such that $k\subset\Pi$ for all $k\in\pclique$.
\end{sentences}
A maximal $\parallel$-clique in the first case will be called a \emph{star direction},
and a \emph{top direction} in the second.

\begin{prop}\label{prop:cliques}
  There are two disjoint classes of maximal $\parallel$-cliques in $\fixout$:
  star directions and top directions.
\end{prop}

\begin{proof}
  Let $\pclique$ be a maximal $\parallel$-clique. Take $k_1, k_2\in\pclique$ with $k_1\neq k_2$.
  These are proper lines that meet in some improper point $a$. In view of \ref{lem:bundle}
  the set $\{k_1, k_2\}$ is never a maximal clique so, there must be a line in $\pclique$
  distinct from $k_1, k_2$. We have two possible cases: there is a line $l\in\pclique$ that
  intersects $k_1, k_2$ in two distinct improper points or not.
  In the first case all the lines of $\pclique$ need to intersect $k_1, k_2$, and $l$ in
  improper points, thus they need to lie on the plane $\gen{k_1, k_2}$.
  In the second case all the lines of $\pclique$ go through $a$.
\end{proof}

The condition:
\begin{itemize}
\item[$(\ast)$]
  every two distinct proper lines, sharing two distinct points with
  lines of a maximal $\parallel$-clique each,
  intersect each other or are parallel.
\end{itemize}
characterizes top directions.

\begin{prop}\label{prop:stars}
  A maximal $\parallel$-clique satisfies $(\ast)$ iff it is a top direction.
\end{prop}

\begin{proof}
  \ltor
  Let $\pclique$ be a maximal $\parallel$-clique satisfying $(\ast)$.
  Suppose to the contrary that $\pclique$ is a star direction.
  By \ref{lem:bundle} take three pairwise distinct non-coplanar lines
  $k_1, k_2, k_3\in\pclique$ and assume that $a\in k_1, k_2, k_3$ is the improper
  point that determines $\pclique$.
  Let $a_1\in k_1$, $a_2\in k_2$ be proper points and
  let $l_1 := \LineOn(a_1, a_2)$.  Now, in view of \ref{fact:basic}\eqref{fact:basic:min},
  take a proper point $b_2\in k_2$ with $b_2\neq a_2$.
  Then, take a proper point $b_3\in k_3$ and $l_2:=\LineOn(b_2, b_3)$.
  Lines $l_1, l_2$ are proper and skew, a contradiction.

  \rtol
  Let $\pclique$ be a top direction and let $l_1, l_2$ be distinct
  proper lines sharing two distinct points with lines of $\pclique$ each.
  Every two lines intersect each other on a projective plane. Hence,
  the lines $l_1, l_2$ share a proper or an improper point, and accordingly,
  $l_1$ intersects $l_2$ or $l_1\parallel l_2$.
\end{proof}

Consequences of \ref{prop:cliques} and \ref{prop:stars} together with
\ref{prop:parallel} are quite essential for our future construction
and will be put down as follows.

\begin{cor}\label{cor:star-directions}
  Star directions are definable in terms of\/ $\fixout$.
\end{cor}

Moreover, every star direction $\pclique$ can be identified with the improper
point shared by all the lines of $\pclique$. We will write $\stardir(a)$ for the
star direction determined by an improper point $a$.

On a side note, in $(\ast)$ we can claim, equivalently, that every two distinct
proper lines, each of which shares two distinct points with lines of a top
direction, are coplanar as planes are definable in $\fixout$ by \ref{cor:planes}.
As for top directions they need not to exist on every plane. If for example
from a plane $\Pi$ we remove a single point, two points, a single line,
a line and a point, or two lines then there is no top direction on $\Pi$.

%% Subsection %%%%%%%%%%%%%%%%%%%%%%%%%%%%%%%%%%%%%%%%%%%%%%%%%%%%%%%%%%%%%%%%%%

\subsection{The main reasoning}\label{sec:main}

We are going to reconstruct the horizon $\W$ in $\fixout$.

In view of \ref{cor:star-directions}, as it has been already stated, to every
improper point $a$ of $\fixproj$ we can uniquely associate the star direction
$\stardir(a)$. Hence we formally have
\begin{equation}
  S = \pointsout \cup S^\infty,
\end{equation}
where
$S^\infty = \bset{ \pclique\colon \pclique \text{ is a maximal $\parallel$-clique that fails $(\ast)$} }$.
The incidence relation of $\fixout$ can be extended now to the set of improper
points as follows. If $a$ is an improper point and $k\in\linesout$, then
\begin{equation}\label{eq:incinfty}
  a\in k :\iff k\in\stardir(a).
\end{equation}
Let us summarize what we have so far.

\begin{prop}
  The structure $\struct{S, \linesout}$ is definable in $\fixout$.
\end{prop}

To get the entire projective space $\fixproj$ we need to recover improper lines
and tell what does it mean that a point, proper or improper, lies on such an
improper line, everything in terms of $\fixout$.

Thanks to \ref{cor:planes} we are allowed to use the term \emph{proper plane} in
the language of $\fixout$. Taking \eqref{eq:incinfty} into account we
can express what does it mean that a proper line is contained in such a plane.
Moreover, by \ref{lem:affine-lines} there is a proper line through an improper point
on every proper plane.
Therefore, for an improper point $a$ and a proper plane $\Pi$ we can make
the following definition
\begin{equation}
  a\in\Pi :\iff (\exists k\in\linesout)\bigl[\, a\in k \Land k\subseteq\Pi \,\bigr],
\end{equation}
where the formula on the right hand side is a sentence in the language of $\fixout$.

Now, for improper points $a, b$ we define \emph{improper adjacency} which means that
$a, b$ lie on an improper line, formally:
\begin{equation}\label{eq:impadjac}
  a\impadjac b :\iff (\nexists\; k\in\linesout)\bigl[\,
    k\in\stardir(a) \Land k\in\stardir(b)\,\bigr].
\end{equation}
Let $\pplanes$ be the class of proper planes in $\fixproj$.
Consider three points $a, b, c$ of $\fixproj$. In view of \ref{impintersec}, with the formula
\begin{multline}\label{collinearity}
  \impcoll(a, b, c) :\iff \impadjac(a, b, c) \Land (\exists \Pi_1, \Pi_2\in\pplanes)
    \bigl[\; \Pi_1\neq\Pi_2 \Land
      a, b, c\in \Pi_1, \Pi_2 \,\bigr]
\end{multline}
we define \emph{improper collinearity} relation.
An improper line in $\fixproj$ is an equivalence class of that relation $\impcoll$ with
two distinct points fixed.
That is, if $a, b$ are two distinct improper points then $[a,b]_{\impcoll}$ is the improper line through $a, b$.
So
\begin{equation}\label{eq:implines}
  \lines^\infty := \bset{ [a,b]_{\impcoll}\colon a,b\in S\setminus\pointsout }
\end{equation}
is the set of all improper lines of $\fixproj$.
Now, the incidence between an improper point $a$
and an improper line $k$ may be formally expressed as follows:
\begin{equation}
  a\in k\iff (\exists b,c\in S)\bigl[\, k = [b,c]_{\impcoll} \Land \impcoll(a,b,c) \,\bigr].
\end{equation}
Finally the underlying projective space
  $$\fixproj = \struct{S,\linesout\cup \lines^\infty}$$
and the horizon $\bstruct{S^\infty, \lines^\infty}$ are both definable
in $\fixout$ which constitutes our main theorem \ref{thm:main-proj}.

%% Section %%%%%%%%%%%%%%%%%%%%%%%%%%%%%%%%%%%%%%%%%%%%%%%%%%%%%%%%%%%%%%%%%%%%%

\section{Partial projective spaces}\label{sec:pps}

Following the idea of partial projective planes founded in \cite{hall}, as well
as the idea of partial affine spaces investigated in \cite{partaffine}, a
\emph{partial projective space} $\pps$ can be defined as a projective space 
$\fixproj=\struct{S, \lines}$ 
with a family $\rlines$ of its lines deleted. Hence $\pps = \struct{S, \ppslines}$,
where $\ppslines = \lines\setminus\rlines$ is the set of undeleted lines.
Further we assume that $\fixproj$ is at least 3-dimensional and
\begin{equation}\label{eq:pps}
  \text{\itshape there is a line in $\ppslines$ through every point on every plane of\/ $\fixproj$.}
\end{equation}
The condition \eqref{eq:pps} is equivalent to requirement that for every point
$a$ and a line $l$ of $\fixproj$ such that $a\notin l$ there is a line in
$\ppslines$ through $a$ that meets $l$.
We are going to recover all the deleted lines strictly in terms of $\pps$.

\begin{thm}\label{thm:pps}
  If \eqref{eq:pps} is satisfied, then every line of\/ $\rlines$ can be recovered
  in\/ $\pps$.
\end{thm}

The proof of this theorem is split into a series of the following lemmas.

\begin{lem}
  There is a triangle of\/ $\pps$ on every plane of\/ $\fixproj$.
\end{lem}

\begin{proof}
  Let $\Pi$ be a plane of $\fixproj$. Take a point $a$ on $\Pi$. By \eqref{eq:pps}
  there is a line $k_1\in\ppslines$ on $\Pi$ through $a$. Now, take a point $b$ on $\Pi$ 
  but not on $k_1$. Again by \eqref{eq:pps} there is a line $k_2\in\ppslines$ on $\Pi$
  through $b$. The lines $k_1, k_2$ meet each other. Finally, take a point $c$ on $\Pi$
  not on $k_1\cup k_2$. By \eqref{eq:pps} we have a line $k_3\in\ppslines$ through $c$ 
  that meets $k_1$ and so, it meets $k_2$ as well.
  The lines $k_1, k_2, k_3$ form the required triangle.  
\end{proof}

\begin{lem}\label{lem:pps-planes}
  Let\/ $\Pi$ be a plane of\/ $\fixproj$ and let\/ $k_1, k_2, k_3$ be the sides
  of a triangle of $\pps$ contained in $\Pi$. Then
  $$\Pi = \Bigl\{ a\in S\colon (\exists k\in\ppslines) \Bigl[\, a\in k 
    \Land \Bigl(\mathop{\wedge}\limits_{i=1}^3 k\cap k_i\neq\emptyset\Bigr) \,\Bigr]\Bigr\}.$$
\end{lem}

\begin{proof}
  $\subseteq\colon$
  Let $a\in\Pi$. If $a$ lies on one of $k_1, k_2, k_3$ we are through. Otherwise,
  by \eqref{eq:pps} there is a line $k\in\ppslines$, through $a$ that meets $k_1$.
  On the projective plane $\Pi$ the line $k$ meets also $k_2, k_3$.
  
  $\supseteq\colon$
  Straightforward.
\end{proof}

\begin{cor}\label{cor:pps-planes}
  Planes of\/ $\fixproj$ are definable in terms of\/ $\pps$.
\end{cor}

Consider a line $l\in\rlines$ and take two planes $\Pi_1, \Pi_2$ of $\fixproj$
that contain $l$. The points on $l$ are not collinear in $\pps$
but they all are contained in both $\Pi_1$ and $\Pi_2$. This, in view of \ref{cor:pps-planes},
lets us define in terms of $\pps$ a ternary collinearity relation whose equivalence classes are the
deleted lines $\rlines$ (cf. \eqref{eq:impadjac}, \eqref{collinearity}, \eqref{eq:implines}).
Now the proof of \ref{thm:pps} is complete.

Nevertheless, it is worth to make a comment regarding the condition
\eqref{eq:pps}. By \ref{cor:pps-planes} all planes containing a line
$l\in\rlines$ are definable, although to recover $l$ two of them are
sufficient.  Hence, one can prove \ref{thm:pps} under an assumption weaker
significantly than \eqref{eq:pps}.  We need to require the existence of a few
points on every line, such that there are lines in $\ppslines$ through these
points, which yield two triangles in distinct planes (see Figure
\ref{fig:cond-star}).  This condition is not so intuitive and as straight as
\eqref{eq:pps}. In the next section we use \ref{thm:pps} to recover improper
lines in a maximal strong subspace of a Grassmann space. From this point of view
it is crucial to note, that all examples of complements in the Grassmann space,
considered in Section~\ref{sec:app}, not fulfilling \eqref{eq:pps} would not
satisfy the new assumption either. For this reason and for simplicity we
abandon the idea of a weaker, though more complex, assumption and continue our
research using \eqref{eq:pps}.

\begin{figure}[!h]
  \begin{center}
    \includegraphics[scale=0.4]{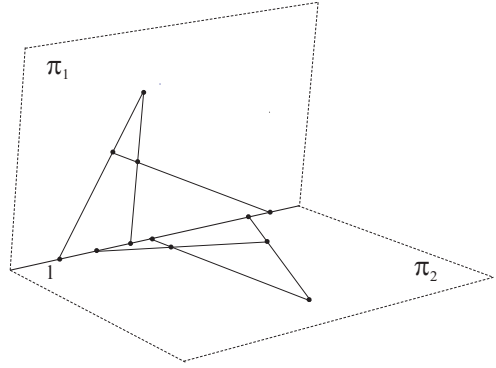}
  \end{center}
  \caption{A weaker version of the assumption \eqref{eq:pps}.}
  \label{fig:cond-star}
\end{figure}

%% Section %%%%%%%%%%%%%%%%%%%%%%%%%%%%%%%%%%%%%%%%%%%%%%%%%%%%%%%%%%%%%%%%%%%%%

\section{Complements in Grassmann spaces}\label{sec:complgrass}

Throughout this section we deal with the Grassmann space $\M = \PencSpace(V, k)$ 
introduced in the beginning, together with a fixed set $\W\subset\Sub_k(V)$. 
More precisely the complement $\fixoutf(\M,\W)$ is investigated.
Two assumptions seem to be essential here.
First of all we assume that the condition \eqref{eq:size-of-line} is satisfied, i.e.\ every
line of $\M$ has  size at least $2\indf(\M,\W) + 2$. To recover improper lines
we assume additionally that
\begin{equation}\label{eq:pps-grass}
  \begin{gathered}
    \text{\itshape if\/ $U$ is a point and $p$ is a line of\/ $\M$ such that}\\
    \text{$U\notin p$ and\/ $U,p$ span some strong subspace, then there is}\\
    \text{a proper maximal strong subspaces through\/ $U$ that meets $p$.}
  \end{gathered}
\end{equation}
Our goal now is to prove an analogue of \ref{thm:main-proj}:

\begin{thm}\label{thm:main-grass}
  If\/ $\M = \PencSpace(V, k)$ is a Grassmann space and $\W$ is its point subset
  satisfying \eqref{eq:size-of-line} and \eqref{eq:pps-grass},
  then both\/ $\M$ and\/ $\W$ can be recovered in the complement\/ $\fixoutf(\M,\W)$.
\end{thm}

Let us start with straightforward consequences of our assumptions.

\begin{lem}\label{lem:prop-line-thr}
  There is a proper line through every point of\/ $\M$.
\end{lem}

\begin{proof}
  Let $U$ be a point of $\M$. In case $U$ is proper every line through
  $U$ is proper. So, assume that $U$ is improper. Take any star or top $X$ 
  through $U$ and any line $p\subset X$. By \eqref{eq:pps-grass} there is 
  a proper maximal strong subspace $Y$ through $U$. So there is a proper
  point $U'\in Y$. As $U\neq U'$ the required line is $\LineOn(U, U')$.
\end{proof}

We used to make a distinction between objects that are contained in the
horizon $\W$ and those that are not. The latter are called proper and the
former improper. 
Following this convention we say that a subspace $X$ of $\M$
is \emph{proper} when $X\not\subset\W$, otherwise it is said to be \emph{improper}.

\begin{prop}\label{prop:stillstrong}
  \strut
  \begin{sentences}
  \item
  If\/ $X$ is a strong subspaces of\/ $\M$ then $X\setminus\W$ is a strong
  subspace of\/ $\fixoutf(\M,\W)$.
  \item
  If\/ $Y$ is a strong subspace of\/ $\fixoutf(\M,\W)$, then there is a strong subspace
  $X$ of\/ $\M$ such that $Y = X\setminus\W$.
  \end{sentences}
\end{prop}

\begin{proof}
  (i)
  Let $U_1, U_2$ be distinct points of $X\setminus\W$. As $X$ is a strong subspace,
  there is a line $p$ in $\M$ that joins $U_1, U_2$. The line $p$ is proper since
  the points $U_1, U_2$ are proper.
  
  (ii)
  The set $Y$ is a clique of adjacency of $\M$.
  Take $X$ to be the maximal strong subspace of $\M$ containing $Y$.
  Note that $Y\subseteq X\setminus\W$ as $Y$ contains no points of the horizon $\W$.
  Now, consider a point $U\in X\setminus\W$. This is a proper point and it is
  collinear in $\M$ with every point of $Y$ as a subset of $X$. All the lines
  through $U$ are proper, so $U\in Y$ as $Y$ is a maximal clique of $\fixoutf(\M,\W)$.
\end{proof}

Note that from the inside of a Grassmann space it is not straightforward to tell whether
a strong subspace is a star or a top (in case $2k=n$ it is even impossible). 
We can make distinction between two types of
strong subspaces however, as two stars (and two tops) are either disjoint or meet
in a point while a star and a top are either disjoint or meet in a line. The
names `star' and `top' come from the outer space, usually a vector space or a
projective space, within which our Grassmann space is defined. So, in view of
\ref{prop:stillstrong}, we say
that a strong subspace $Y$ of $\fixoutf(\M,\W)$ is a star (a top) if there is a star 
(respectively a top) $X$ in $\M$ such that $Y = X\setminus\W$.

\begin{lem}\label{fact:rec-strong}
  Let\/ $X$ be a strong subspace of\/ $\M$, then
  $X$ and\/ $\W\cap X$ can be recovered in $\fixoutf(X,{\W\cap X}) = \fixoutf(\M,\W) | X$.
\end{lem}

\begin{proof}
  First of all, note that $\indf(X,{\W\cap X}) \le \indf(\M,{\W})$ and
  every line of $X$ has size at least $2\indf(X,{\W\cap X}) + 2$.
  Since every strong subspace $X$ of $\M$ carries the structure of a projective
  space, we can apply \ref{thm:main-proj} locally in $X$.
\end{proof}

Let $U\in\W$. Consider the family $M_U$ of all proper maximal strong subspaces through
$U$ in $\M$. Since every line of $\M$ is extendable to a star and to a top, in view of
\ref{lem:prop-line-thr} the set $M_U$ has 
at least two elements.
As we prove in Section~\ref{sec:main}, every improper point $U$ in a projective space
$X\in M_U$ can be uniquely identified with the bundle $\stardirin(U,X)$ (star direction) of parallel 
lines that pass through $U$ and are contained in $X$. 
Let 
\begin{cmath}
  S^{\ast\infty} := \Big\{\stardirin(U,X)\colon X\in M_U, \; U\in\W\Big\}.
\end{cmath}
We introduce the relation $\mathord{\uzg}\subseteq S^{\ast\infty}\times S^{\ast\infty}$ 
as follows:
in case $X_1, X_2$ are of different types (a top and a star up to ordering)

\begin{multline}\label{uzg:startop}
  \stardirin(U_1, X_1)\uzg \stardirin(U_2, X_2) :\iff 
    U_1, U_2 \in X_1\cap X_2 \Land \\
    \bigl(\exists p_1\in\stardirin(U_1, X_1)\bigr)\bigl(\exists p_2\in\stardirin(U_2, X_2)\bigr) 
    \bigl[\, p_1, p_2\neq X_1\cap X_2 \Land  \\
    \topof(p_1)\cap\starof(p_2) \text{ or } \starof(p_1)\cap\topof(p_2) 
\text{ is a proper line distinct from } X_1\cap X_2\,\bigr]
\end{multline}
and in case $X_1, X_2$ are of the same type (two stars or two tops)
\begin{multline}\label{uzg:starstar}
  \stardirin(U_1, X_1)\uzg \stardirin(U_2, X_2) :\iff 
    (\exists U\in\W)
    (\exists X \in M_U \\ \text{ of the other type than } X_1,X_2) 
      \bigl[\, \stardirin(U_1, X_1)\uzg \stardirin(U, X) \uzg \stardirin(U_2, X_2) \,\bigr].
\end{multline}
The relation $\uzg$ lets us identify improper points shared by a top and a star.

\begin{prop}\label{prop:uzgadniacz}
  Let $X_1, X_2$ be two proper maximal strong subspaces of different types in $\M$, and let
  $U_1,U_2\in \W$. The following conditions are equivalent:
  \begin{sentences}
  \item\label{first}
    $U_1=U_2\in X_1\cap X_2$,
  \item\label{second}
    $\stardirin(U_1, X_1)\uzg \stardirin(U_2, X_2)$.
  \end{sentences}
\end{prop}

\begin{proof}
  Without loss of generality we can assume that $S := X_1$ is a star and
  $T := X_2$ is a top. 
  
  (i)$\implies$(ii):
  Let $U_1=U_2=:U$.
  As $U\in S\cap T$, by \ref{fact:lineinstartop},
  $l := S\cap T$ is a line of $\M$.
  From \ref{lem:bundle} there are lines $p_1\in\stardirin(U_1,S)$ and 
  $p_2\in\stardirin(U_2,T)$ distinct from $l$.
  Note that $U\in\topof(p_1)\cap\starof(p_2)$. Hence 
  $\topof(p_1)\cap\starof(p_2)$ is always a line $p$ in $\M$. Assume to the contrary 
  that for all $p_1, p_2\neq l$ every such a line $p$ is improper.
  
  Consider a line $p_1\in\stardirin(U, S)$ with $p_1\neq l$. Note that $U\in T\cap\topof(p_1)$.
  Take the set of proper planes 
    $${\cal F} = \{ \Pi\colon l\subset\Pi\subset T\}.$$
  If $l$ is proper, then all the planes containing $l$ are proper. If $l$ is improper, then
  we apply \ref{impintersec}. In both cases
  the cardinality of $\cal F$ is at least $\indf(\M,\W)+2$.
  
  Let $\Pi\in{\cal F}$. In view of
  our assumptions all of the lines $p=\starof(p_2)\cap\topof(p_1)$, for any line
  $p_2$ through $U$ contained in $\Pi$, are improper.
  By \ref{lem:top-star-top} these lines lie on the plane $f(\Pi)$ contained in $\topof(p_1)$.
  Now, the question is if any of planes $f(\Pi)$ is proper. To answer this question consider
  the set $f({\cal F})$. By \ref{lem:top-star-top}(ii) $f({\cal F})$ is the set of planes that share a single
  line $l'$. In case $l'$ is proper all the planes in $f({\cal F})$ are proper.
  In case $l'$ is improper, by \ref{impintersec}, there are proper planes in $f({\cal F})$.
  Any way we can take a proper plane $\Pi'\in f({\cal F})$.  
  Every line intersecting all of these lines on $\Pi'$ would have more than $\indf(\M,\W)$
  improper points, a contradiction.

  (ii)$\implies$(i):
  Assume to the contrary that $U_1\neq U_2$.
  Note, by \eqref{uzg:startop}, that $l := S\cap T$ is a line.
  From \eqref{uzg:startop} we have $U_1\in\topof(p_1)$, $U_2\in\starof(p_2)$, and $U_1\adjac U_2$. 
  By \ref{cor:starcaptop}, $l\subset\topof(p_1)$ or $l\subset\starof(p_2)$.
  Thus $\topof(p_1)=\topof(l)=T$ or $\starof(p_2)=\starof(l)=S$. 
  In both cases $S$ and $T$ share two distinct lines, a contradiction.
\end{proof}
We need to identify improper points common for two tops or two stars as well.

\begin{prop}\label{prop:uzgadniacz1}
  Let $X_1, X_2$ be two proper maximal strong subspaces of the same type in $\M$, and let
  $U_1,U_2\in \W$. The following conditions are equivalent:
  \begin{sentences}
  \item\label{first1}
    $U_1=U_2\in X_1\cap X_2$,
  \item\label{second1}
    $\stardirin(U_1, X_1)\uzg \stardirin(U_2, X_2)$.
  \end{sentences}
\end{prop}
\begin{proof}
  We restrict our proof for two stars $S_1:=X_1$, $S_2:=X_2$, as the case involving two tops is dual.

  (i)$\implies$(ii):
  Let $U_1=U_2=:U\in S_1\cap S_2$. By \ref{lem:prop-line-thr} we take a proper line $l$ through $U$.
  Then the top $T(l)$ is proper and $U\in S_1\cap S_2\cap T(l)$. 
  Thus, from (i)$\implies$(ii) of \ref{prop:uzgadniacz} we obtain $\stardirin(U, S_1)\uzg \stardirin(U, T)$ and
  $\stardirin(U, T)\uzg \stardirin(U, S_2)$, that in view of \eqref{uzg:starstar} means 
  that  $\stardirin(U, S_1)\uzg \stardirin(U, S_2)$.

  (ii)$\implies$(i):
  Straightforward by (ii)$\implies$(i) of \ref{prop:uzgadniacz}.
\end{proof}

Recovering in \ref{thm:main-grass} goes in the following two steps:

\paragraph{Step I}

Applying results of Section~\ref{sec:main}, we recover the points of $\W$
and improper lines of $\M$ that are contained in proper  stars or tops. The
condition stated in \ref{lem:prop-line-thr} is critical here. 

In view of \ref{prop:uzgadniacz} and \ref{prop:uzgadniacz1} the relation $\uzg$ is an equivalence
relation. 
Thanks to \ref{lem:prop-line-thr} we have $M_U\neq\emptyset$ for
every improper point $U$ (cf. \ref{exm:neighbourhood}). This makes it possible to cover the complement
$\fixoutf(\M,\W)$ with the family of proper stars and tops. Now,
\ref{fact:rec-strong} can be applied for each member of that covering
to recover points of the horizon $\W$.
Therefore
\begin{cmath}
  \W = S^{\ast\infty}/\mathord{\uzg}.
\end{cmath}
Now, let us denote by $\peki^\infty$ the set of all improper lines of $\M$.
If $X$ is a proper maximal strong subspace of $\M$, then we write $\peki_X^\infty$
for the set of all improper lines in $X$.
Every line of $\peki_X^\infty$ can be defined as a section
of two proper planes, like it was done in \eqref{collinearity}, \eqref{eq:implines}.
Thus 
\begin{multline*}
  \peki^{\ast\infty} := \bigcup\Big\{
    [\pclique]_{\uzg}\colon \pclique = \stardirin(U,X),\; U\in l \in\peki_X^\infty,\\
        X\text{ is a proper maximal strong subspace of } \M \Big\}
\end{multline*}        
is the set of improper lines of $\M$ that are recoverable in $\fixoutf(\M,\W)$
by means of Section~\ref{sec:main}.

\paragraph{Step II}

Note that $\peki^\infty\setminus\peki^{\ast\infty}\neq\emptyset$, as
long as both $S(l)$ and $T(l)$ are improper for some line $l$.
We recover such lines applying results of Section~\ref{sec:pps}.
It would not be possible without assumption \eqref{eq:pps-grass} (cf. \ref{exm:twointervals}).

Let us focus on the improper star $S(l)$ which, due to dualism, 
does not cause loss of generality.
Points of $S(l)$ are recovered, as all points in $\W$ are already recovered.
So $S(l)$ can be considered as a projective space with some lines deleted, 
i.e. a partial projective space.
Observe that the assumption \eqref{eq:pps} is fulfilled in $S(l)$ by \eqref{eq:pps-grass}.
Hence, \ref{thm:pps} can be applied, and thus all the remaining unrecovered lines of $S(l)$ 
can be recovered now: the line $l$ in particular.
This way every line from the set
\begin{cmath}
  \peki^\infty\setminus\peki^{\ast\infty} = 
   \bigl\{l\in\peki^\infty\colon T(l),\; S(l) \text{ are improper}\bigr\} 
\end{cmath}
can be recovered in some improper star. 

Finally, we recover all lines from $\peki^\infty$,  which makes the proof of
\ref{thm:main-grass} complete.

%% Section %%%%%%%%%%%%%%%%%%%%%%%%%%%%%%%%%%%%%%%%%%%%%%%%%%%%%%%%%%%%%%%%%%%%%

\section{Applications}\label{sec:app}

\begin{exm}[Slit space]\label{exm:slit}
  Let $\fixproj$ be a projective space and $\W$ its subspace.
  It is clear that $\indf(\fixproj,\W) = 1$ and thus
  $2\indf(\fixproj,\W)+2 = 4$. So, if lines of $\fixproj$ have size at least 4 
  (i.e. the ground field is not $\GF(2)$) the condition \eqref{eq:size-of-line} 
  is satisfied. Hence we can apply \ref{thm:main-proj} and recover
  $\fixproj$ and $\W$ in the complement $\fixoutf(\fixproj,\W)$.
\end{exm}

\begin{exm}[Two-hole slit space]\label{exm:sulima}
  Again, let $\fixproj$ be a projective space, but now let $\W = \W_1\cup\W_2$, where
  $\W_1, \W_2$ are complementary subspaces of $\fixproj$ and none of $\W_1, \W_2$ is a hyperplane.
  A line of $\fixproj$, not entirely contained in $\W$, meets $\W$ in none, one or two points.
  In \cite{sulima} the complement ${\goth T} = \struct{\mathscr{S}, \mathcal{T}}$, where $\mathscr{S}$
  coincides with the point set of $\fixoutf(\fixproj,\W)$ while $\mathcal{T}$ is the set of lines
  that miss $\W$, i.e.\ those lines of $\fixoutf(\fixproj,\W)$ which have no improper points.
  So, the geometry of two-hole slit space ${\goth T}$ differs a bit from our complement $\fixoutf(\fixproj,\W)$
  in that the later has more lines. It has been proved in \cite{sulima} that 
  $\fixproj$ can be defined in terms of $\goth T$. In particular, all the lines 
  of $\fixoutf(\fixproj,\W)$ has been defined in $\goth T$. Consequently
  the complement $\fixoutf(\fixproj,\W)$ can be recovered in $\goth T$.
  If we additionally assume that lines of $\fixproj$ are of size at least 6, then \eqref{eq:size-of-line}
  is satisfied and we can apply \ref{thm:main-proj}.
  Therefore, in these settings $\fixproj$ and $\W$ can be recovered in $\fixoutf(\fixproj,\W)$.
  %and thus in $\goth T$.
  Though the result of \cite{sulima} is stronger, as there are no requirement on the size of lines, 
  our result embraces larger class of applications.
\end{exm}

\begin{exm}[Multi-hole slit space]\label{exm:multihole}
  Let us generalize the examples \ref{exm:slit} and \ref{exm:sulima} fixating
  a family $W_1,\dots, W_m$ of subspaces in the projective space $\fixproj$.
  Take $\W =\bigcup_{i=1}^m W_i$ and note that $\indf(\fixproj,\W)= m$ if there is a line 
  intersecting every of $W_i$ in precisely one point and all these points are pairwise distinct.
  Therefore, in general $\indf(\fixproj,\W)\leq m$.
  The minimal size of a line, that is required in order to apply \ref{thm:main-proj}, remains
  $2\indf(\fixproj,\W)+2$ and it is determined by arrangement of subspaces $W_1,\dots,W_m$.
  Such minimal size of a line varies from 4, as it is in the case of a slit space, up to $2m+2$.
\end{exm}

\begin{exm}[The complement of a quadric (i.e.\ of a polar space)]\label{exm:quadric}
  Now, let $\polarity$ be a polarity on a projective space $\fixproj$.
  In vein of linear algebra $\fixproj$, if it is desarguesian,
  corresponds to a projective space over a, say left, vector space over a division
  ring $F$ and the polarity $\polarity$ corresponds to a non-degenerate reflexive
  sesqui-linear form. Denote by $Q$ the set of all selfconjugate points w.r.t.\ $\polarity$
  i.e.\ the quadric determined by $\polarity$. It coincides with the point set of $\fixproj$
  when $\polarity$ is symplectic, or it is a specific case of a so called 
  \emph{quadratic set}, that is, a set of points with the property that every line 
  which meets $Q$ in at least three points is entirely contained in $Q$. 
  We are interested in the latter.
  So, $\indf(\fixproj,Q) = 2$ and if every line of $\fixproj$ has size at least 6
  ($F\neq\GF(2), \GF(3), \GF(4)$ in terms of linear algebra)
  the condition \eqref{eq:size-of-line} is satisfied. Therefore, in view of
  \ref{thm:main-proj} we see that $\fixproj$ and $Q$ can be recovered in the complement
  of $Q$ in $\fixproj$.
\end{exm}

\begin{exm}[The complement of a Grassmann subspace in a Grassmann space]\label{exm:interval}
  In \cite{slitgras} complements of interval substructures in Grassmannians where
  investigated. Such substructures are unique in that they bear the structure of 
  Grassmannians and only those have this property. In Grassmann spaces interval
  subspaces play an analogous role, as it was shown in \cite{correl}, and this is the
  reason to call them Grassmann subspaces. 
  So, consider a projective Grassmann space $\M = \PencSpace(V, k)$ and its
  interval subspace $\W := [Z, Y]_k$ for some subspaces $Z, Y$ of $V$.
  In our construction we require that $\W \neq\Sub_k(V)$, c.f. \ref{fact:basic}\eqref{fact:basic:notall},
  that is either $Z\neq\Theta$ ($\Theta$ being the trivial subspace of $V$), or
  $Y\neq V$.
  So, we have $\indf(\M,\W) = 1$ as $\W$ is a subspace of $\M$, hence for ground fields of $V$
  different from $\GF(2)$ the condition \eqref{eq:size-of-line} is satisfied.
  Observe that if $\W$ contains a maximal top, then $Z=\Theta$, 
  and if $\W$ contains a maximal star, then $Y=V$. Therefore, for every line
  $p$ of $\M$ one of its maximal strong extensions $\starof(p)$ or $\topof(p)$ is proper. 
  This means that \eqref{eq:pps-grass} holds true and thus $\M$ together with $\W$ can
  be recovered applying \ref{thm:main-grass}.
\end{exm}

\begin{exm}[The complement of a polar Grassmann space]\label{exm:polargras}
  Let $\A$ be a polar space embedded into some projective space $\fixproj$. 
  The polar Grassmann space $\PencSpace(\A, k)$ (cf. \cite{polargras}) is 
  embeddable into the projective Grassmann space $\M = \PencSpace(\fixproj, k)$. 
  Take $\W$ to be the point set of the former. Again
  $\indf(\M,\W) = 1$ or $\indf(\M,\W) = 2$ 
  as $\A$ is a null-system (i.e.\ $\W$ is a subspace of $\M$) or not, 
  respectively. So, if we drop the case where lines of
  $\fixproj$ are of size 3 or of sizes 3, 4, 5, respectively, the condition
  \eqref{eq:size-of-line} is fulfilled. Now, take a point $U$ and a line $p$ 
  of $\M$ such that $U\notin p$ and $U, p$ span a strong subspace. Note that
  every point on $p$ is adjacent to $U$. So, take a line $q$ through $U$ that meets $p$.
  The star $\starof(q)$ can not be contained in $\W$ (cf. \cite{polargras})
  which yields that the condition \eqref{eq:pps-grass} is fulfilled. Therefore,
  \ref{thm:main-grass} can be applied to recover $\M$ and $\W$.
  
  Note that in case $\A$ is a null-system this is an example of removing a
subspace from a Grassmann space,  likewise in Example~\ref{exm:interval}.
\end{exm}

So far we have seen examples where Theorem~\ref{thm:main-proj} or \ref{thm:main-grass} 
is applicable, sometimes under certain additional assumptions. The following examples 
show that the condition \eqref{eq:size-of-line} or \eqref{eq:pps-grass} is not always 
satisfied and our theorems cannot be applied.

\begin{exm}[Spine space]
  Following \cite{spinesp}, \cite{autspine} a spine space $\A:=\SpineSp(k,m,V,W)$,
  where $V$ is a vector space, $W$ its fixed subspace, and $k, m$ some fixed integers, 
  can be considered as the Grassmann space $\M = \PencSpace(V, k)$
  with the set of points 
  \begin{cmath}
    \W=\{ U\in\Sub_k(V)\colon \dim(U\cap W)\neq m\}
  \end{cmath}  
  removed. The lines of $\A$ are those lines of $\PencSpace(V, k)$ with at 
  least 2 proper points. 
  If $0 < m$, $m-1\neq k$, and $m-1\neq \dim(W)$, i.e.\ if there are
  affine lines in $\SpineSp(k,m-1,V,W)$, then each of these lines, from view of $\M$, 
  has a unique proper point.
  This contradicts \ref{fact:basic}\eqref{fact:basic:min} and means that
  the condition \eqref{eq:size-of-line} is not satisfied
  in spine spaces in general. 
  If however, $m=0$, as in the case of linear complements, or there are no
  affine lines in  $\SpineSp(k,m-1,V,W)$, then $\indf(\A,\W) = 1$ and, under
  assumption that the ground field of $V$ is not $\GF(2)$, the condition
  \eqref{eq:size-of-line} is fulfilled.
  The condition \eqref{eq:pps-grass} is problematic as it will be shown
  in \ref{exm:lincomp}.
  
  Nevertheless, in \cite{autspine} it was shown that with more sophisticated
  methods the ambient space $\M$ can always be recovered in $\A$.
  These methods consist in iterated, step by step, recovering the horizon
  of the horizon, treated as a spine space. 
\end{exm}

\begin{exm}[Space of linear complements]\label{exm:lincomp}
  Two papers \cite{blunck-havlicek} and \cite{lincomp} deal with the structure
  of complementary subspaces to a fixed subspace in a projective space
  $\fixproj$, though in completely different settings. This structure is
  embeddable into an affine space. Following \cite{lincomp} it arises as a specific
  case of a spine space, that is, we can think of it as of a Grassmann space
  $\M$, over a vector space $V$, with the set $\W$ of those points of
  $\M$ that are not linear complements of some fixed subspace $W$ of $V$.
  By \cite{lincomp}, \cite{spinesp}, or \cite{autspine} we have $\indf(\fixproj,\W) = 1$
  so, if we rule out $\GF(2)$ we have \eqref{eq:size-of-line} satisfied.
  If we take however, a point $U$ of $\M$ with $2\le\dim(U\cap W)$, then 
  there is no proper point $U'$ adjacent to $U$. This violates \ref{lem:prop-line-thr}
  and hence \eqref{eq:pps-grass} does not hold true.
\end{exm}

\begin{exm}[Affine Grassmann space]\label{exm:affinegr}
  Now let us start with a Grassmann space $\M = \PencSpace(V, k)$ with its geometric hyperplane $\W$
  removed. By \cite{cuypers} the complement of\/ $\W$ in $\M$ is an affine
  Grassmann space, or an affine Grassmannian $\A$.
  Assume that the ground field is not $\GF(2)$.
  In this case $\indf(\M,\W) = 1$ as every line of $\M$
  either meets $\W$ in a point, or is entirely contained in $\W$. Hence
  $2\indf(\M,\W)+2 = 4$ which means that the condition \eqref{eq:size-of-line}
  is fulfilled. 
  Following \cite{cuypers}, note however that $\W$ is the set of complementary
  $k$-subspaces to a fixed subspace $W$ of codimension $k$ in $V$, or it is the set
  of $k$-subspaces killed by some $k$-linear alternating form on $V$. In the first
  case we get the structure of linear complements as in  \ref{exm:lincomp}, so
  \eqref{eq:pps-grass} fails here either.
\end{exm}

\iffalse
\begin{exm}[Ruled quadric]
  Let $Q$ be a ruled quadric in a projective space $\fixproj=\struct{S,\Lines}$ and
  take $\W := S\setminus Q$. The proper lines of $\fixproj$ are the generators of $Q$
  and there are lines in $\fixproj$ with up to two proper points. Our requirement
  \eqref{eq:size-of-line} is not satisfied, but it has been shown in \cite{moirer}\warning
  however, that $\fixproj$ can be recovered from $Q$ using different methods.
\end{exm}
\fi

The next two examples may seem artificial but they are intended to point out
the substance and necessity of the condition in \ref{lem:prop-line-thr} and 
of \eqref{eq:pps-grass} respectively, so they are minimal in the sense that 
$\W$ is possibly small while \ref{lem:prop-line-thr} or \eqref{eq:pps-grass} 
is not satisfied.

\begin{exm}[The complement of the neighbourhood of a point]\label{exm:neighbourhood}
  Consider a Grassmann space $\M = \PencSpace(V, k)$. Let $U$ be a point of
  $\M$. Take the set $\W$ of all points of $\M$ collinear with $U$. As far as
  adjacency $\adjac$ is concerned we can say that $\W$ is the
  \emph{neighbourhood} of $U$. Since $\M$ is a gamma space, $\W$ is a subspace of
  $\M$ (cf. \cite{cohen}). This implies that $\indf(\M,\W) = 1$ but there 
  are no proper lines through $U$, so the condition in \ref{lem:prop-line-thr}
  fails. This makes impossible to recover the horizon $\W$ pointwise.
\end{exm}

\begin{exm}[The complement of two large interval subspaces]\label{exm:twointervals}
  In a Grassmann space $\M = \PencSpace(V, k)$, where $1<k<n=\dim(V)$,
  consider a line $l=\penc(H, B)$ and its two extensions: the star $S = [H)_k$
  and the top $T = (B]_k$. 
  Let $U\in S$ with $U\notin l$. Every top through $U$ that intersects $l$ 
  intersects also $S$ in a unique line (a line through $U$ that intersects $l$).
  All these new lines span a plane $[H, U+B]_k$. 
  Hence, $\W = (U+B]_k\cup [H)_k$ does not satisfy \eqref{eq:pps-grass}.
  
  Dually, for a point $U\in T$ with $U\notin l$ we have a similar plane in the top 
  $T$, namely $[U\cap H, B]_k$. This time, $\W = (B]_k\cup [U\cap H)_k$  does not 
  satisfy \eqref{eq:pps-grass}.
\end{exm}

Note that \eqref{eq:pps-grass} fails in \ref{exm:twointervals} while
\ref{lem:prop-line-thr} is still valid. This example resembles \ref{exm:sulima} 
and \ref{exm:multihole} in that two or more subspaces are removed. 
Using our methods one can
recover the ambient projective space with any number of its subspaces removed,
while in Grassmann spaces removing only two subspaces can make recovery
impossible or at least seriously much more complex.

The examples above show that most of the time it is possible to recover ambient
projective space from a complement of its point subset as the only requirement is
the number of proper points on a line.
As projective Grassmann spaces can be covered by two families of projective spaces
it is tempting to utilize this and apply the forementioned result. It is not 
straightforward though. Our procedure relies on partial projective spaces and
this introduces a new requirement: the condition \eqref{eq:pps-grass}, which
turns out to be quite restrictive as many classical examples fail to satisfy it.
Maybe another approach would be more suitable in case of Grassmann spaces.

%% References %%%%%%%%%%%%%%%%%%%%%%%%%%%%%%%%%%%%%%%%%%%%%%%%%%%%%%%%%%%%%%%%%%

%% Contact %%%%%%%%%%%%%%%%%%%%%%%%%%%%%%%%%%%%%%%%%%%%%%%%%%%%%%%%%%%%%%%%%%%%%

\begin{flushleft}
  \noindent
  K. Petelczyc, M. \.Zynel\\
  Institute of Mathematics, University of Bia{\l}ystok,\\
  Akademicka 2, 15-267 Bia{\l}ystok, Poland\\
  \verb+kryzpet@math.uwb.edu.pl+,
  \verb+mariusz@math.uwb.edu.pl+
\end{flushleft}

\end{document}